\documentclass[amsthm]{article}

\usepackage[T2A]{fontenc}
\usepackage[utf8x]{inputenc}

\usepackage{amssymb}
\usepackage{amsthm}
\usepackage{mathtools}
\usepackage{tikz}
\usepackage{mathtools}

\theoremstyle{definition} 
\newtheorem{thm}{Theorem}    
\newtheorem{remark}[thm]{Remark}  
\newtheorem{lemma}[thm]{Lemma}

\newtheorem{proposition}[thm]{Proposition}

\newcommand{\RR}{\mathbb R}
\newcommand{\ZZ}{\mathbb Z}
\newcommand{\QQ}{\mathbb Q}
\newcommand{\CC}{\mathbb C}

\newcommand{\Length}{\operatorname{Length}}
\newcommand{\e}{\varepsilon}

\title{Tropical formulae for summation over a part of $SL(2,\mathbb Z)$}

\author{Nikita Kalinin\footnote{National Research University Higher School of Economics, Soyuza Pechatnikov str., 16, St. Petersburg, 190121, Russian Federation. Support from the Basic Research Program of the National Research University Higher School of Economics is gratefully acknowledged. Supported in part by Young Russian Mathematics award.}, Mikhail Shkolnikov\footnote{IST Austria. Klosterneuburg 3400, Am campus 1. Supported by ISTFELLOW program.}}
\begin{document}
\maketitle

%%it is better to have more space before the formulae
\medskip

Let $f(a,b,c,d)=\sqrt{a^2+b^2}+\sqrt{c^2+d^2}-\sqrt{(a+c)^2+(b+d)^2}$, let $(a,b,c,d)$ stand for $a,b,c,d\in\ZZ_{\geq 0}$ such that $ad-bc=1$. Define
\begin{equation}
\label{eq_main}
F(s) = \sum_{(a,b,c,d)} f(a,b,c,d)^s.
\end{equation} 
In other words, we consider the sum of the powers of the triangle inequality defects for the lattice parallelograms (in the first quadrant) of area one.

We prove that $F(s)$ converges when $s>1$ and diverges at $s=1/2$. ({\bf This papers differs from its published version: Fedor Petrov showed us how to easily prove that $F(s)$ converges for $s>2/3$ and diverges for $s\leq 2/3$, see below.}) 
We also prove 
$$\sum\limits_{(a,b,c,d)} \frac{1}{(a+c)^2(b+d)^2(a+b+c+d)^2} = 1/3,$$

%$$\sum\limits_{\substack{(a,b,c,d),\\ 1\leq a\leq 2b,1\leq c\leq 2d}}\frac{1}{b^2d^2(b+d)^2}=2/3,$$
% it s sort of the same.

and show a general method to obtain such formulae. The method comes from the consideration of the tropical analogue of the caustic curves, whose moduli give a complete set of continuous invariants on the space of convex domains.

%Since this article is a consequence of the conference on Easter island, whose history is a cautionary tale, we find it reasonable to include our reflection on modern mathematics as Addendum.

{\bf keywords: tropical geometry, summation, SL(2,Z), pi

}

\section{A generalization of Archimedes' exhaustion method}

\rightline{\it Гужем узырлэн ужасьёссэ кема ужатэмез потэм.}
\rightline{\it Кенеш сётэмзыя, со пась дӥся, сапег кутча, писпуэ тубе}
\rightline{\it но шундыез саникен кутыса улэ, медаз, пе, пуксьы.}
\rightline{\it Ӝыны нунал гинэ чидаз.}
\rightline{\it (An excerpt from an Udmurt folklore).} 
\rightline{ A merchant desired that peasants work more.}
\rightline{ Somebody gave him an advice: climb a tree in a fur coat and valenki,}
%\rightline{ }
\rightline{ and try to retain the sun with a pitchfork, preventing sunset.}
\rightline{ Only half a day he withstood there.}

We find this epigraph complementing the history of Easter Island: it is easy in destroy a civilization simply trying to maximize formal parameters such as the height of a moai, for example.

%Богач хотел, чтобы крестьяне летом больше работали. По совету, он надел тулуп, валенки, залез на дерево и стал вилами держать солнце, чтобы не садилось. Только полдня продержался.

We have already described the following ideas in \cite{pi_short}. Here we give more detailed and motivated exposition.
\subsection{Triangles formed by tangent lines}

Let $p_1,q_1,p_2,q_2,\in\RR, p_1q_2-p_2q_1=D>0$.
Consider three lines given by %a concave curve $C$ and its three tangent lines orthogonal to directions $(p_1,q_1),(p_2,q_2),(p_1+p_2,q_1+q_2)$, given by equations 
\begin{align*}
l_1(x,y)=&p_1x+q_1y+c_{p_1,q_1}=0,\\
l_2(x,y)=&p_2x+q_2y+c_{p_2,q_2}=0,\\
l_3(x,y)=&(p_1+p_2)x+(q_1+q_2)y+c_{p_1+p_2,q_1+q_2}=0.\\
\end{align*}
%as in Figure~\ref{fig_tangents}. 
Let %$A_i$ be the tangent point of $i$-th tangent line to $C$ and 
$A_{ij}$ be the intersection point of $i$-th and $j$-th lines. We compute

$$A_{12}=\frac{1}{D}(c_{p_2,q_2}q_1-c_{p_1,q_1}q_2, c_{p_1,q_1}p_2-c_{p_2,q_2}p_1)$$
$$A_{13}=\frac{1}{D}(c_{p_1+p_2,q_1+q_2}q_1-c_{p_1,q_1}(q_1+q_2), c_{p_1,q_1}(p_1+p_2)-c_{p_1+p_2,q_1+q_2}p_1)$$
$$A_{23}=\frac{1}{D}(c_{p_1+p_2,q_1+q_2}q_2-c_{p_2,q_2}(q_1+q_2), c_{p_2,q_2}(p_1+p_2)-c_{p_1+p_2,q_1+q_2}p_2).$$

By direct computation we obtain
\begin{lemma}
\label{lemma_main}
The area of the triangle $A_{12}A_{13}A_{23}$ is $$\frac{(c_{p_1,q_1}+c_{p_2,q_2}-c_{p_1+p_2,q_1+q_2})^2}{2D}.$$ The lattice length of each of sides $A_{12}A_{23},A_{12}A_{13},A_{13}A_{23}$ with respect to the lattice $\ZZ(p_1,q_1)+\ZZ(p_2,q_2)$ is $$\frac{1}{D}(c_{p_1,q_1}+c_{p_2,q_2}-c_{p_1+p_2,q_1+q_2}).$$ Moreover, this number is equal to the value of $l_1,l_2,l_3$ at the point where all of them are equal.
\end{lemma}

Recall that the lattice length of an interval $I$ in a direction of a primitive lattice vector $(p,q)\in\ZZ^2$ is the Euclidean length of $I$ divided by $\sqrt{p^2+q^2}$.  If both endpoints of $I$ belong to $\ZZ^2$, then the lattice length of $I$ is the number of lattice points on $I$ minus one. Define the lattice length of a curve as the sum of the lattice lengths of its parts which are intervals with rational slope. For example, if a curve has no such intervals, its lattice length is zero. With the same rule we define all these notions for any $\ZZ^2$ lattice in $\RR^2$.

%Let $(p_1,q_1),(p_2,q_2)$ be primitive $\ZZ^2$ lattice vectors in the first quadrant (i.e. $p_1,p_2,q_1,q_2,\geq 0$) and $p_1q_2-p_2q_1=1$.

\subsection{Main technical theorem}

Let $f:[x_0,x_1]\to \RR$ be a concave continuous non-negative function. Denote $A_2=(x_0,f(x_0)), A_1=(x_1,f(x_1))$. Consider two tangent lines to the plot of $f$ given by $P_2x+Q_2y+c_{P_2,Q_2}=0$ (tangent at $(x_0,f(x_0))$) and $P_1x+Q_1y+c_{P_1,Q_1}=0$ (tangent at $(x_1,f(x_1))$) such that $P_1Q_2-P_2Q_1=D>0$  %Suppose also that 
%$$P_2x_0+Q_2f(x_0)+c_{P_2,Q_2}=0=P_1x_1+Q_1f(x_1)+c_{P_1,Q_1}=0,$$
and $P_ix+Q_if(x)+c_{P_i,Q_i}\leq 0$ for $i=1,2, x\in [x_0,x_1]$.

 Denote by $S$ the area in between of these lines and the graph of $f$, see Figure~\ref{fig_tangents}. Let $A_{12}$ be the intersection point of these two lines. Denote by $L$ the sum of the lattice length of two intervals $A_1A_{12},A_2A_{12}$ minus the lattice length of the curve $(x,f(x)), x_0\leq x \leq x_1$.

Let $SL^{+}(2,\ZZ)$ be the submonoid of the group $SL(2,\ZZ)$, generated by 
$$\begin{pmatrix}
1&1\\
0&1\\
\end{pmatrix}\text{ and }\begin{pmatrix}
1&0\\
1&1\\
\end{pmatrix}.
$$

\begin{thm}
\label{th_main}
Let pairs $(p_1,q_1),(p_2,q_2)$ runs by 
$\begin{pmatrix}
p_1&q_1\\
p_2&q_2\\
\end{pmatrix}\in SL^{+}(2,\ZZ)\cdot\begin{pmatrix}
P_1&Q_1\\
P_2&Q_2\\
\end{pmatrix}
.$ The following equalities hold.  
$$\frac{1}{2D}\sum (c_{p_1,q_1}+c_{p_2,q_2}-c_{p_1+p_2,q_1+q_2})^2=S,$$
$$\frac{1}{D}\sum (c_{p_1,q_1}+c_{p_2,q_2}-c_{p_1+p_2,q_1+q_2})=L,$$

where  $c_{p,q}$ is chosen such that $px+qy+c_{p,q}=0$ is a tangent line to the graph of $f$.

\end{thm}
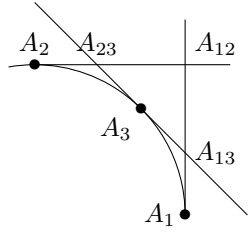
\begin{figure}[h]
\begin{center}
\begin{tikzpicture}[scale=2]
\draw (1,0) arc (0:100:1);
\draw (1,0)--(1,1.3);
\draw (0,1)--(1.3,1);
\draw (0,1.4142)--(1.4142,0); 

\draw (1,0) node{$\bullet$};
\draw (0,1) node {$\bullet$};
\draw (1/1.4142,1/1.4142) node {$\bullet$};

\draw (1,0) node[left] {$A_1$};
\draw (0,1) node[above] {$A_2$};
\draw (1/1.4142,1/1.4142) node[below left] {$A_3$};
\draw (1,1) node[above right] {$A_{12}$};
\draw (0.4142,1) node[above] {$A_{23}$};
\draw (1,0.4142) node[right] {$A_{13}$};

\end{tikzpicture}
\end{center}
\caption{$A_2=(x_0,f(x_0)),A_1=(x_1,f(x_1)), (p_1,q_1)=(1,0),(p_2,q_2)=(0,1)$}
\label{fig_tangents}
\end{figure}

\begin{proof}
This theorem follows from repetitive application of Lemma~\ref{lemma_main}. We start with two tangents $p_1x+q_1y+c_{p_1,q_1}=0$ and $p_2x+q_2y+c_{p_2,q_2}=0$ (at points $A_2,A_1$) and apply the lemma for $(p_1,q_1),(p_2,q_2)$ and there sum $(p_1+p_2,q_1+q_2)$. Then we apply the lemma for the tangent lines at pair of points $A_3,A_1$ and at $A_2,A_3$, etc.
\end{proof}

\subsection{Example: obtaining $\pi$}
\label{ex_pi}
%\begin{example} 
If $f(x)=\sqrt{1-x^2}$ and $x_0=0,x_1=1$, then the area in between of tangent lines to the circle at $(0,1)$ and $(1,0)$ and the circle is $1-\pi/4$. Also, it is easy to see that $c_{p,q}=\sqrt{p^2+q^2}$ in this case. Thus, for $$f(a,b,c,d)=\sqrt{a^2+b^2}+\sqrt{c^2+d^2}-\sqrt{(a+c)^2+(b+d)^2}$$ Theorem~\ref{th_main} gives the identities:
\begin{equation}
\label{eq_main2}
F(2) = \sum f(a,b,c,d)^2=2(1-\pi/2),\ \ 
F(1) = \sum f(a,b,c,d)^1=2
\end{equation}
where the sum runs by all $a,b,c,d\in\ZZ_{\geq 0}$ such that $ad-bc=1$.  The second equality is due to the fact that the sum of the lattice lengths of two tangent intervals $A_2A_{12},A_1A_{12}$ is $2$, see \cite{pi_short}.
%\end{example}

\subsection{Example: quadratic functions}
%\begin{example}
 Consider the curve given by the equation $y=1-\mu x^2$ with $\mu>0$. For $p,q\in\ZZ_{\geq 0}$ the line $px+qy+c_{p,q}=0$ is tangent to this curve if $px+q(1-\mu x^2)+c_{p,q}=0$ has a double root, i.e. its discriminant $p^2+4q\mu(c_{p,q}+q)$ is zero. Therefore, $c_{p,q}=-\frac{4q^2\mu+p^2}{4q\mu}$.  

Then, $-c_{p_1,q_1}-c_{p_2,q_2}+c_{p_1+p_2,q_1+q_2}$ in this case is equal to $$\frac{4q_1^2\mu+p_1^2}{4q_1\mu}+\frac{4q_2^2\mu+p_2^2}{4q_2\mu} - \frac{4(q_1+q_2)^2\mu+(p_1+p_2)^2}{4(q_1+q_2)\mu},$$ which is, in turn,
$$\frac{(4q_1^2\mu+p_1^2)q_2(q_1+q_2)+(4q_2^2\mu+p_2^2)q_1(q_1+q_2)-(4(q_1+q_2)^2\mu+(p_1+p_2)^2)q_1q_2}{4\mu q_1q_2(q_1+q_2)}$$

$$=\frac{p_1^2q_2^2+p_2^2q_1^2-2p_1p_2q_1q_2}{4\mu q_1q_2(q_1+q_2)} = \frac{D^2}{4\mu q_1q_2(q_1+q_2)}.$$

Then, at point $(\frac{1}{\sqrt\mu},0)$ the equation of the tangent line at this point is $2\sqrt\mu x+ y = 2$. Consider the lattice generated by $(2\sqrt\mu,1),(0,1)$. In this case $D=2\sqrt\mu$. Let $$\begin{pmatrix}
p_1&q_1\\
p_2&q_2
\end{pmatrix}\in SL^+(2,\ZZ)\cdot\begin{pmatrix}
2\sqrt\mu&1\\
0&1
\end{pmatrix}
$$

Therefore $$\sum \frac{D^3}{2\cdot16\mu^2q_1^2q_2^2(q_1+q_2)^2} + \int_0^{1/\sqrt \mu}(1-\mu x^2) = \frac{1}{2}\left(\frac{1}{\sqrt\mu}+\frac{1}{2\sqrt\mu}\right),$$

this gives

$$
\sum \frac{1}{4q_1^2q_2^2(q_1+q_2)^2} =\frac{1}{12}.
$$

Therefore the dependence on $\mu$ is mysteriously eliminated! What should not surprise us:  the expression depends only on second coordinates of the vectors $(p,q)$ and we start with vectors $(0,1),(2\sqrt\mu,1)$. It is the same as to start with vectors $(1,1),(0,1)$. Therefore we can rewrite it as

$$\sum_{\substack{p_1,q_1,p_2,q_2\geq 0\\ p_1q_2-p_2q_1=1\\
q_i\geq p_i}} \frac{1}{q_1^2q_2^2(q_1+q_2)^2} = \frac{1}{3}$$

Then, after a change of coordinates, 
%using that $\begin{pmatrix}
%1&1\\
%0&1\\
%\end{pmatrix}=\begin{pmatrix}
%1&0\\
%0&1\\
%\end{pmatrix}\cdot\begin{pmatrix}
%1&1\\
%0&1\\
%\end{pmatrix}
%$
we can write
$$\sum_{\substack{p_1,q_1,p_2,q_2\geq 0\\ p_1q_2-p_2q_1=1}} \frac{1}{(p_1+q_1)^2(p_2+q_2)^2(p_1+q_1+p_2+q_2)^2} = \frac{1}{3}$$

%\end{example}

%{\bf Standard lattice.}
%\begin{example} Consider the same curve $y=1-\mu x^2, \mu>1/4$. Then we consider the horizontal tangent line and $x+y+c_{11} = 0$ and generate the lattice by $(0,1),(1,1)$. thus we obtain
%
%\add{forulae}
%
%\end{example}

\subsection{Example: triangle with irrational slope}

Take a positive irrational $\alpha$ and let $f(x) = \alpha-\alpha x$ and $x_0=0,x_1=1$. We start with the graph of $f$ and support lines $x-\alpha=0,y-1=0$.

Then, we obtain the following formula $$\alpha = \sum_{n=0}^\infty (p_n-\alpha q_n)^2r_{n+1},$$
$$\alpha +1= \sum_{n=0}^\infty (p_n-\alpha q_n)r_{n+1},$$
 where $r_0,r_1,\dots$ are numbers in the continued fraction for $\alpha$  and $\frac{p_n}{q_n}$ is the $n$-th convergent to $\alpha$. We use the convention 
 $(p_0,q_0)=(1,0), (p_1,q_1) = ([\alpha],1)$, $r_1=[\alpha]=\frac{p_1}{q_1}, r_1= [\frac{1}{\alpha-r_1}], \frac{p_2}{q_2} = r_1+\frac{1}{r_2}$, etc.  
 %%\zeta of a number

%\subsection{Example: hyperbola}
%Let us see what happens when we consider the curve $(y-1)(x-1) = n$. In this case let $f(x) = \frac{n}{x-1}+1,x_0=-n+1, x_1=0,f(x_0)=0, f(x_1)=-n+1$. Then, $dy(x-1)+dx(y-1)=0$, so if we have a tangent line $px+qy+c_{pq}=0$ then $pdx+qdy=0$, so $\frac{p}{q} = \frac{y-1}{x-1}$, so $\frac{p}{q}(x-1)^2=n$, so $$x=-\sqrt{\frac{nq}{p}}+1, y=-\sqrt{\frac{np}{q}}+1,$$ and finally $c_{pq}=-px-qy=2\sqrt{npq}-p-q$. 
%
%Therefore, $$\sum_{\substack{p_1,q_1,p_2,q_2\geq 0\\ p_1q_2-p_2q_1=1}}$$

\subsection{Convergence of $F$}
We thank Yves-Fran\c{c}ois Petermann, Fedor Petrov and Fedor Nazarov who provided us with the ideas of this section. Denote $||a,b||=\sqrt{a^2+b^2}$.

\begin{lemma}
\label{eq_fEst}
For $f=||a,b||+||c,d||-||a+c,b+d||, ad-bc=1$ we have
\begin{equation}
\label{eq_fEstimation}
f(a,b,c,d)=\frac{2}{\big(||a,b||+||c,d||+||a+c,b+d||\big)\cdot\big(||a,b||\cdot ||c,d||+ac+bd\big)}.
\end{equation}

\end{lemma}
\begin{proof}
Indeed, 
$$f(a,b,c,d)\big(||a,b||+||c,d||+||a+c,b+d||\big)=$$ 
$$= \big(||a,b||+||c,d||\big)^2-||a+c,b+d||^2 = $$
$$= 2(\sqrt{(a^2+b^2)(c^2+d^2)}-ac-bd),$$
which, in turn, after multiplication by $\sqrt{(a^2+b^2)(c^2+d^2)}+ac+bd$ and accompanying by the fact that $ad-bc=1$ gives the desired estimate.
\end{proof}
\begin{lemma}
For $s>1$ the series $F(s)$ converges.
\end{lemma}
\begin{proof}
Using the symmetry $f(a,b,c,d) = f(d,c,b,a)$ we may rewrite 
$$\sum f(a,b,c,d)^s = 2\sum_{||a,b||\geq ||c,d||} f(a,b,c,d)^s - f(1,0,0,1)^s.$$
Using \eqref{eq_fEstimation} we see that $f(a,b,c,d)\leq \frac{1}{||a,b||^{2}}$.% if $||a,b||\geq ||c,d||$.  
Then, for each $||a,b||\geq ||c,d||$ the vector $(c-a,d-b)$ does not belong to $\ZZ_{\geq 0}^2$, thus we can write
\begin{equation}
\label{eq_esti}
\sum f(a,b,c,d)^s \leq 2\sum_{(a,b)\in\ZZ_{\geq 0}^2} \frac{1}{||(a,b)||^{2s}}.
\end{equation}
Therefore, if we denote  by $r_2(n)$ the number of presentation of $n$ as the sum of two squares of natural numbers, we obtain 
$\sum f(a,b,c,d)^s\leq \sum r_2(n) n^{-s}$ which is equal to $4\zeta(s) L_{-4}(s)$ (see, for example, Theorem 306 in \cite{hardy1979introduction}) which is a finite number for $s>1$.
\end{proof}

\begin{lemma}
For $s>1$ all the derivatives of $F$ converge.
\end{lemma}
\begin{proof}
Indeed, $k$-th derivative of $\sum f^s$ produces the series  $\sum (\log f)^k f^s$. We repeat all the steps in the proof of the precedent lemma and note that by \eqref{eq_fEstimation} we have that $$\frac{1}{||a,b||^{3}}\leq f(a,b,c,d)\leq \frac{1}{||a,b||^{2}},$$ so multiplication by $\log f$ changes \eqref{eq_esti} by adding $C\log(||a,b||)^k$ to the nominator which has no effect on convergence.
\end{proof}

\begin{lemma} 
For $s\leq 1/2$ the series $F(s)$ diverges.
\end{lemma}

\begin{proof}
It is enough to prove that $\sum_{k=1}^{\infty} f(1,0,k,1)^s\geq const\cdot\sum 1/k^{2s}$.
Indeed, by  \eqref{eq_fEstimation} we see that $f(1,0,k,1)\geq \frac{2}{5k\cdot 3k}$.

%Let $a,b,c,d$ as in the summation, $||a,b||> ||c,d||$ and $||c,d||>||(a-c,b-d)||$. We consider 
%$F_{c,d}(s) = \sum_{k=1}^{\infty} f(a+kc,b+kd)^s$. It is enough to prove that $F_{c,d}(s)\geq \frac{1}{32||c,d||^{3s}}\sum 1/k^{2s}$.
%
%Note that $||a,b||< 2||c,d||$ by construction. Then, using \eqref{eq_fEstimation},
%$$f(a+kc,b+kd) = $$
%
%$$\frac{1}{k^2}\frac{2}{||(a/k+c,b/k+d)||+||(c/k,d/k)||+||(a/k+c+c/k, b/k+d+d/k)||}\cdot$$
%
%$$\cdot\frac{1}{||(a/k+c,b/k+d)|| ||(c,d)||+(a/k+c)c+(b/k+d)d} \geq $$
%
%$$\geq \frac{1}{k^2}\frac{2}{3 ||c,d||+||c,d||+ 4||c,d||}\cdot \frac{1}{2||c,d||||c,d||+2||c,d||||c,d||} = $$
%
%$$=\frac{1}{32 k^2 ||(c,d)||^3}.$$
\end{proof}

{\bf Fedor Petrov's comment:} It is easy to prove that $F(s)$ converges for $s>2/3$ and diverges for $s\leq 2/3$. Indeed, for two vectors $x=(a,b),y=(c,d)$ we have that $f(a,b,c,s)=f(x,y)$ is of order $\frac{1}{|x||y|^2}$ if $|x|<|y|$. Therefore,
$$\sum f(x,y)^s\sim \sum_x \frac{1}{|x|^s}\sum_{k=0}^\infty \frac{1}{|y_0+kx|^{2s}}\sim$$
$$\sim  \sum_x \frac{1}{|x|^s}\sum_k \frac{1}{|x|^{2s}}\frac{1}{k^{2s}}\sim \sum_x\frac{1}{|x|^{3s}}\sim \sum_n r_2(n)n^{-\frac{3}{2}s},$$
and the latter converges when $\frac{3}{2}s>1$ and diverges when $\frac{3}{2}s\leq 1$.

\begin{remark}
Computer experiments show that $F(s)$ can be analytically extended to complex $s$ with $\mathrm{Re}\  s< 1/2$, but we do not know how to prove it, due to quick growth of derivatives. For example, the values of $F$ and its seven derivatives at $s=1/2+3i$ equal (appoximately)
$$-1.34-0.88i,34.4+9.8i,-839-186i,20653+3430i,-513272-58439i,$$ 
$$1.2\cdot 10^7+8.3\cdot 10^5i,-3.2\cdot10^8-6.3\cdot 10^6i,8.2\cdot10^9-1.9\cdot10^8i.$$
\end{remark}

\section{Integral affine invariants of convex domains}

Below we give a conceptual explanation for the described exhaustion method. To a compact convex domain $\Omega$ we associate a tree $C_\Omega\subset\Omega$ having a canonical structure of a rational tropical analytic curve. This curve describes $\Omega$ completely: any invariant of convex domains depends on the lengths of edges of a corresponding curve, i.e. becomes a function of its moduli. The moduli describe the sizes of triangles that we cut out in the exhaustion, see Figure \ref{fig_blowdownc}. We give general formulas for area, perimeter and lattice perimeter in terms of $C_\Omega$ (section \ref{sec_uniformls}). The summations presented in the previous section are, therefore, seen as applications of these formulas to some simple shapes.

The tropical curve $C_\Omega$ is given by the tropical series $F_\Omega,$ which is the lattice distance to $\partial\Omega.$ Its non-zero level sets $\Omega_t=F_\Omega^{-1}[t,\infty)$ are convex polygons with rational slopes of edges (we call them $\QQ$-polygons), we may think of these level sets as of propagation of wave front, and $C_\Omega$ becomes the tropical analog of its caustic. 

Sending $\Omega$ to $\Omega_t, t>0,$ defines the canonical evolution (see \cite{misha_thesis}) on the space of convex domains. It pushes general convex domains to $\QQ$-polygons and, more specifically, to Delzant polygons in case when $\Omega$ is already Delzant, see Theorem \ref{thm_deltedlz}, or if $\partial\Omega$ has no corners, see Theorem \ref{thm_omegldelz}. 

Recall that a $\QQ$-polygon $\Delta$ is called Delzant if for any pair of adjacent sides primitive vectors in their directions form a basis in $\ZZ^2.$ These polygons, up to translations and $SL(2,\ZZ),$ are in one-to-one correspondence with compact smooth symplectic toric surfaces, see \cite{MR984900}. To a smooth surface $X$ one associates $\Delta=\mu(X),$ where $\mu\colon X\rightarrow \RR^2$ is the moment map of the Hamiltonian torus action on $X$ (see \cite{MR1853077}). The image $\Delta$ is a Delzant polygon and its vertices are the images of fixed points of the torus action. The intuition behind the exhaustion is that cutting a corner of $\Delta$ corresponds to a symplectic blow-up of $X$ at the fixed point of the torus action. In particular, the exhaustion (as well as the canonical evolution) gives a way to define a toric surface corresponding to a general compact convex domain in the category of Hamiltonian pro-manifolds.

\subsection{Tropical caustic of Delzant polygon}
\label{sec_curvepol}

First we give a mechanical description of the tropical curve $C_\Delta$ in the case of Delzant polygon $\Delta.$ We are going to introduce a finite system of particles in the polygon and, as a set, $C_\Delta$ will be the trajectory traced by these particles. In the initial position, we set a particle at every vertex of $\Delta$ and send it towards the interior with the velocity vector equal to the sum of the primitive vectors spanning the sides adjacent to the vertex. Every particle moves rectilinearly until some group of particles collides. For example, in case of $\Delta$ equal to a square with horizontal and vertical sides all four particles meet at the center. Similarly, if $\Delta$ is a Delzant triangle all three particles meet at the same time (see Figure \ref{fig_sqtr}) and the tropical curve $C_\Delta$ has a unique vertex at the point of collision, i.e. the tropical curve is {\it unbranched}. We note that in both cases the polygons correspond to Fano toric surfaces $X=\CC P^1\times\CC P^1$ or $\CC P^2$ with a multiple of anti-canonical polarisation. This is true in general, $C_\Delta$ is has a unique vertex if and only if Delzant polygon $\Delta$ is proportional to a lattice polygon with a unique lattice point in the interior. Indeed, by putting the origin to the vertex of $C_\Delta$ and rescaling $\Delta$ by the inverse of the total time of the evolution we get such a lattice polygon. The total time of the rescaled polygon is equal to $1$ and the initial positions of particles are opposite to their velocities. 

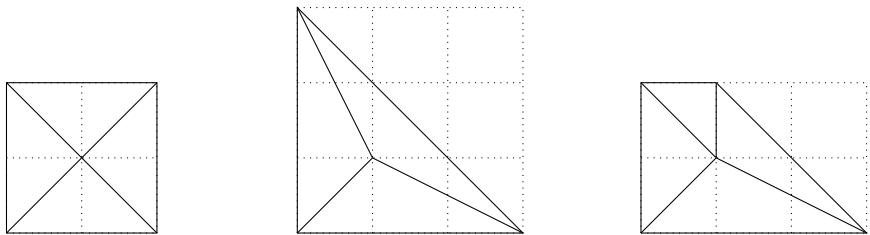
\begin{figure}[h]
\begin{tikzpicture}

\draw[black,step = 1.0,thin, dotted](0,0) grid (2,2);
\draw(0,0)--++(0,2)--++(2,0)--++(0,-2)--++(-2,0);
\draw(0,0)--++(1,1);
\draw(1,1)--++(1,1);
\draw(1,1)--++(-1,1);
\draw(1,1)--++(1,-1);

\begin{scope}[xshift=110]
\draw[black,step = 1.0,thin, dotted](0,0) grid (3,3);
\draw(0,0)--++(0,3)--++(3,-3)--++(-3,0);
\draw(1,1)--++(-1,-1);
\draw(1,1)--++(-1,2);
\draw(1,1)--++(2,-1);
\end{scope}

\begin{scope}[xshift=240]
\draw[black,step = 1.0,thin, dotted](0,0) grid (3,2);
\draw(0,0)--++(3,0)--++(-2,2)--++(-1,0)--++(0,-2);
\draw(1,1)--++(-1,-1);
\draw(1,1)--++(-1,1);
\draw(1,1)--++(2,-1);
\draw(1,1)--++(0,1);
\end{scope}

\end{tikzpicture}
\caption{Some Delzant polygons with unbranched tropical curve.}

\label{fig_sqtr}
\end{figure}

If we perform a blowup of $\CC P^2$ at one of the three fixed points of the torus action, this corresponds to cutting a corner of its moment triangle. Note that the size of the cut may vary which corresponds to a different choice of symplectic structure on the blowup. If the integral of the new symplectic form along the exceptional divisor is small enough, the particles corresponding to the new corners will collide first forming a new particle with the velocity equal to the sum of velocities of collided particles. Note that this velocity is the same as of the particle which would be emitted from the blown-up corner. Moreover, after the collision the three remaining particles behave in the same way as if they would be simply emitted from the corners of the triangle without the blowup. In particular, they meet at the same terminal point. In this case the tropical curve has two vertices (see Figure \ref{fig_trbl}). 

While the cut gets bigger, eventually all the four initial particles will be again meeting at the same time (see Figure \ref{fig_sqtr} on the right) and the curve has one vertex which is terminal for the process. Larger cut produces trapezium in which the particles are meeting in pairs. After the two collisions (apparently happening at the same time and at the same altitude) two particles going horizontally towards each other. They collide at the terminal point (see Figure \ref{fig_trbl}).

\begin{figure}[h]
\begin{tikzpicture}
\draw[black,step = 1.0,thin, dotted](0,0) grid (4,3);
\draw(0,0)--++(0,3)--++(1,0)--++(3,-3)--++(-4,0);
\draw(0,0)--++(1.333,1.333);
\draw(1.333,1.333)node{$\bullet$};
\draw(4,0)--++(-2.666,1.333);
\draw(1.333,1.333)--++(-0.333,0.666)--++(0,1);
\draw(0,3)--++(1,-1);

\begin{scope}[xshift=200]
\draw[black,step = 1.0,thin, dotted](0,0) grid (4,2);
\draw(0,0)--++(0,2)--++(2,0)--++(2,-2)--++(-4,0);
\draw(0,0)--++(1,1)--++(-1,1);
\draw(4,0)--++(-2,1)--++(-1,0);
\draw(2,2)--++(0,-1);
\draw(1.5,1.3)node{$2$};
\draw(1.5,1)node{$\bullet$};
\end{scope}

\end{tikzpicture}

\caption{Different symplectic structures on a Hirzebruch surface produce different curves on its moment polygon. Note that the terminal point (marked by $\bullet$) of the polygons is not integral (although the polygons are lattice) and the curve on the right has an edge of multiplicity two. The terminal point on the left is the lower vertex and on the right it is the middle of the multiplicity two segment. Note that in both pictures there is a pair of sides of $\Delta$ and a pair of edges of $C_\Omega$ which are parallel (see Remark \ref{rem_exppar})} 
\label{fig_trbl}
\end{figure}
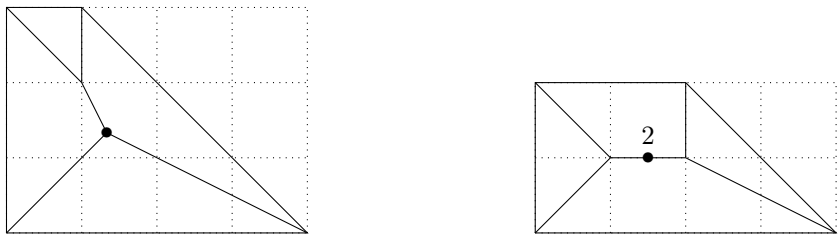

The rule of particle collision requires a minor clarification. In order to do that, we need to introduce masses of particles which are positive integers. The initial masses of particls are all taken to be one and set to the vertices of $\Delta.$ A momentum of a particle is its mass times velocity and velocity is postulated to be primitive for all existing particles.  In the case when a group of particles collide they give a new particle such that the momentum is preserved. For example, in Figure \ref{fig_trbl} on the right we see a pair of mass two particles annihilating each other at the terminal point of the mechanical process. The conservation of momentum is also known as the balancing condition in tropical geometry, see Figure \ref{fig_balancing}.

\begin{lemma}
For any Delzant polygon $\Delta$ all the particles stay inside $\Delta^\circ$ and there exists a a unique (terminal) time at which a group of particles annihilate (without emission of a new particle) each other at the terminal point $p_\Delta.$ 
\label{lem_termination} 
\end{lemma}

In particular, the trajectory of all particles $C_\Delta$ is a weighted finite planar graph with rational slopes. The one-valent vertices of $C_\Delta$ are the vertices of $\Delta.$ The weights (or multiplicities) on edges correspond to the masses of particles tracing them. We note that the moment preservation is equivalent to the balancing condition shown on Figure \ref{fig_balancing}. This turns $C_\Delta$ in to a tropical curve. An edge of a tropical curve is called multiple if its weight is greater than two. One can give a full description of multiple edges in the Delzant case.

 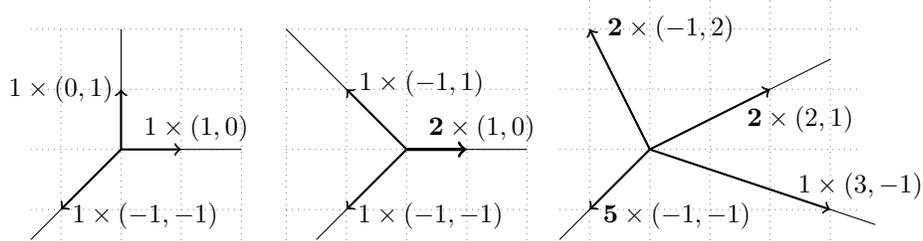
\begin{figure}[h]
    \centering

\begin{tikzpicture}[scale=0.4]
\draw(0,0)--++(4,0);
\draw(0,0)--++(0,4);
\draw(0,0)--++(-3,-3);
\draw[->][thick](0,0)--++(2,0);
\draw[->][thick](0,0)--++(0,2);
\draw[->][thick](0,0)--++(-2,-2);
\draw(2.5,0.7)node{$1\times (1,0)$};
\draw(-2,2)node{$1\times (0,1)$};
\draw(0.8,-2.2)node{$1\times (-1,-1)$};
\draw[black,step = 2.0,very thin, dotted](-3,-3) grid (4,4);

\begin{scope}[xshift=270]
\draw(0,0)--++(-4,4);
\draw(0,0)--++(-3,-3);
\draw(0,0)--++(4,0);
\draw[->][very thick](0,0)--++(2,0);
\draw[->][thick](0,0)--++(-2,2);
\draw[->][thick](0,0)--++(-2,-2);
\draw(2.5,0.7)node{${\bf 2}\times (1,0)$};
\draw(0.5,2.2)node{$1\times (-1,1)$};
\draw(0.8,-2.2)node{$1\times (-1,-1)$};
\draw[black,step = 2.0,very thin, dotted](-4,-3) grid (4,4);
\end{scope}

\begin{scope}[xshift=500]
\draw(0,0)--++(-3,-3);
\draw(0,0)--++(-2,4);
\draw(0,0)--++(6,3);
\draw(0,0)--++(7.5,-2.5);
\draw[->][thick](0,0)--++(-2,-2);
\draw[->][thick](0,0)--++(-2,4);
\draw[->][thick](0,0)--++(4,2);
\draw[->][thick](0,0)--++(6,-2);
\draw(0.9,-2.2)node{${\bf 5}\times (-1,-1)$};
\draw(7,-1.2)node{$1\times (3,-1)$};
\draw(5,1)node{${\bf 2}\times (2,1)$};
\draw(0.7,4)node{${\bf 2}\times (-1,2)$};
\draw[black,step = 2.0,very thin, dotted](-3,-3) grid (7,5);
\end{scope}

\end{tikzpicture}
\caption{Examples of balancing condition in local pictures of tropical curves near vertices. The notation ${\bf w}\times (p,q)$ means that the corresponding edge has the weight ${\bf w}$ and the primitive vector $(p,q).$ The vertex on the left picture is {\it smooth} (i.e. has multiplicity one), the vertices in the middle and right pictures are not smooth having multiplicities two and forty, see Figure \ref{fig_balancingdual}}
\label{fig_balancing}

\end{figure}

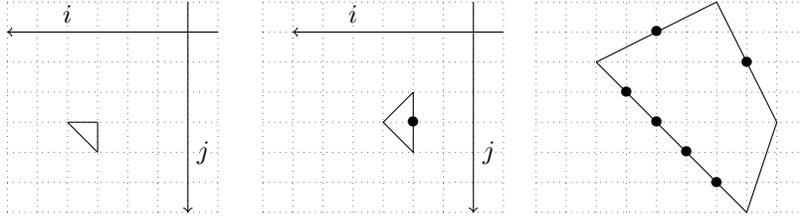
\begin{figure}[h]
    \centering

\begin{tikzpicture}[scale=0.4]
\draw(4,3)[->]--++(-7,0);
\draw(3,4)[->]--++(0,-7);
\draw (-1,3) node[above] {$i$};
\draw (3,-1) node[right]{$j$};

\draw(0,0)--++(-1,0);
\draw(-1,0)--++(1,-1);
\draw(0,-1)--++(0,1);

\draw[black,step = 1.0,very thin, dotted](-3,-3) grid (4,4);

\begin{scope}[xshift=270]

\draw(4,3)[->]--++(-7,0);
\draw(3,4)[->]--++(0,-7);
\draw (-1,3) node[above] {$i$};
\draw (3,-1) node[right]{$j$};

\draw(1,1)--++(-1,-1)--++(1,-1)--++(0,2);

\draw(1,0)node{$\bullet$};
\draw[black,step = 1.0,very thin, dotted](-4,-3) grid (4,4);
\end{scope}

\begin{scope}[xshift=500]

\draw(-1,2)--++(5,-5)--++(1,3)--++(-2,4)--++(-4,-2);
\draw[black,step = 1.0,very thin, dotted](-3,-3) grid (6,4);
\draw(0,1)node{$\bullet$};
\draw(1,0)node{$\bullet$};
\draw(2,-1)node{$\bullet$};
\draw(3,-2)node{$\bullet$};

\draw(1,3)node{$\bullet$};

\draw(4,2)node{$\bullet$};
\end{scope}

\end{tikzpicture}
\caption{Polygons dual to local models of tropical curves on Figure \ref{fig_balancing}. A dual polygon is defined up to a translation, its lattice points $(i,j)$ correspond to the monomials $ix+jy+a_{ij}$ (in some polynomial defining a curve) which contribute to the value of the tropical polynomial at the vertex. Note that we need to reverse coordinate axes because of the ``inf'' (instead of more conventional here ``sup'') agreement  in \eqref{eq_fomega}, in the definition of the tropical curve. Sides of polygons are orthogonal to the edges of curves. Moreover an integral length of a side, computed as one plus the number of lattice points in its interior, is the weight of the corresponding dual edge. The areas of the polygons are $1/2$, $1$ and $20$, thus the multiplicities of the dual edges are $1,$ $2$ and $40.$}
\label{fig_balancingdual}
\end{figure}

\begin{proposition}
If the terminal point $p_\Delta$ is a vertex of $C_\Delta$ for Delzant polygon $\Delta$ then $C_\Delta$ has no multiple edges. If the terminal point is not a vertex then it belongs to the middle of the {\it last} edge $last_\Delta$ appearing in the building of $C_\Delta,$ the last edge $last_\Delta$ has multiplicity $2$ and it is the only multiple edge of $C_\Delta.$ The vertices that are not $p_\Delta$ or ends of $last_\Delta$ are $3$-valent and smooth (see Figures \ref{fig_balancing} and \ref{fig_balancingdual}).   
\label{prop_lastedge} 
\end{proposition}

The curve $C_\Delta$ contains a lot of geometric information. At the most basic level we have.

\begin{proposition}
Consider a compact smooth symplectic toric manifold $X$ and an irreducible boundary divisor $D.$ Then to this divisor one associates a side $s=\mu(D)$ of the Delzant polygon $\Delta=\mu(X).$ Let $s_1$ and $s_2$ be the edges of $C_\Delta$ coming out of the ends of $s.$ Then $s$ together with continuations of $s_1$ and $s_2$ form a Delzant triangle iff the self-intersection of $D$ in $X$ is $-1$. 
\label{prop_eblowdun}
 \end{proposition}
 
 \begin{proof}
In the case when $D$ is the exceptional divisor, the only possible configuration is shown on the Figure \ref{fig_blowdownc}.

 On the other hand, we note that all Delzant triangles are the same up to $SL(2,\ZZ),$ rescaling and translations. Therefore, if $v_1$ and $v_2$ are the velocities of the particles sent from the ends of $s\subset\partial\Delta$ then $s$ is parallel to $v_1-v_2$ (see Figure \ref{fig_blowdownc} on the left). If we denote by $w_1$ and $w_2$ the primitive vectors in the directions of sides adjacent to $s$ then by the construction $v_1=w_1+(v_1-v_2)$ and $v_2=w_2+(v_2-v_1).$ Thus, $w_1=v_2$ and $w_2=v_1$ and the prolongations of sides intersect at a point $p_0$. Denote by $\tilde\Delta$ the convex hull  of $\Delta$ and $p_0$. In this case $X$ is realized as a symplectic blow-up of a manifold with the moment polygon $\tilde\Delta$ and $D$ is the exceptional divisor. 
 \end{proof}

\begin{figure}
\begin{tikzpicture}

\draw(3.2,3.2)node{$s$};

\draw[->](4,2)--++(-1,0);
\draw(3.1,1.7)node{$v_2$};

\draw[->](2,4)--++(0,-1);
\draw(1.7,3.1)node{$v_1$};

\draw[->](2,2)--++(-1,-1);
\draw(0.5,1.3)node{$v_1+v_2$};

\draw[->](4,2)--++(0,-1);
\draw(3.3,1)node{$w_2=v_1$};

\draw[->](2,4)--++(-1,0);
\draw(1.2,3.7)node{$w_1=v_2$};

\draw(0,4)--++(2,0)--++(2,-2)--++(0,-2);
\draw(2,4)--++(0,-2)--++(-2,-2);
\draw(4,2)--++(-2,0);

\begin{scope}[xshift=200]

\draw(0,4)--++(4,0)--++(0,-4);
\draw(4,4)--++(-4,-4);
\draw[dashed](2,4)--++(2,-2)--++(-2,0)--++(0,2);

\end{scope}

\end{tikzpicture}
\caption{Two particles collide, trajectories form a Delzant triangle with the side $s,$ emanating a new one with velocity $v_1+v_2$. We will show that two sides of the polygon are parallel to the trajectories of the particles, if they collide. We can move $s$ towards the intersection of the sides. This corresponds to blowdown on the right. Note that blowing up gives a simple branching of the tropical curve. %And the resulting vertex is smooth (non-multiple).
}
\label{fig_blowdownc}
\end{figure}
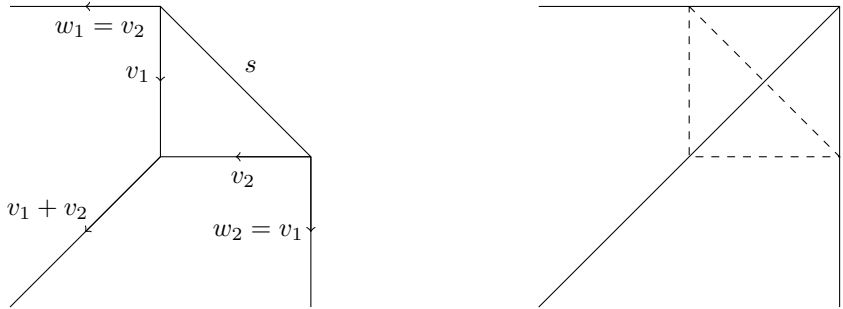

If the conditions of Proposition \ref{prop_eblowdun} are satisfied, we call the side $s$ removable. Removing a side is opposite to cutting a corner.

\begin{remark}
A more visual criterium for the removability of $D$ is that the sides adjacent to $s$ are parallel to the two edges of $C_\Delta$ adjacent to $s$ in the reversed order (as on the Figures \ref{fig_trbl} and \ref{fig_blowdownc}).
\label{rem_exppar}
\end{remark}

By Proposition \ref{prop_lastedge}, unless $C_\Delta$ has only one vertex  or it has only two vertices belonging to the ends of $last_\Delta$ (and in this cases we call $C_\Delta$ unbranched), there exist a side $s$ such that the particles sent from the ends of this side meet in a non-terminal vertex, such $s$ satisfies conditions of Proposition \ref{prop_eblowdun}. In this case we can easily relate $C_\Delta$ and the curve corresponding to $\Delta$ with $s$ blown down (see Figure \ref{fig_blowdownc}).

One may notice that in the cases when $\Delta$ is a lattice polygon, $C_\Delta$ passes through the vertices of $\Delta'=\operatorname{ConvexHull} (\Delta^\circ\cap\ZZ^2).$ Iterating the process we get a sequence of polygons $\Delta,\Delta',\Delta'',\dots$ And we can think of a function $F_\Delta$ on $\Delta$ whose level sets are $\partial\Delta,\partial\Delta',\partial\Delta'',\dots$ for the levels $0,1,2,\dots$ This function $F_\Delta$ will be piecewise linear and not smooth exactly along $C_\Delta,$ and we say that $F_\Delta$ defines $C_\Delta.$ In general, a formal definition of $F_\Delta$ can be given as follows.

\begin{remark}
After the first few moments, a particle sent from a vertex of $\Delta$ stay at the same distance from the sides adjacent to that vertex. 
\label{rem_initdist}
\end{remark}

The distance is taken in the ``integral'' (lattice-invariant) normalization, a length of a lattice vector $v=(x,y)$ is defined to be $|gcd(x,y)|,$ this is equal to the quotient of $|v|$ by the length of a primitive vector parallel to $v.$  Therefore, a particle runs a unit of distance through the unit of time in this normalization. In particular, if we want to compute the distance between a line with a rational slope and a point, one applies an $SL(2,\ZZ)$ transformation to make the line horizontal and the integral normalized distance between the new point and the new line (is the same as between the old ones) is equal to the usual distance in this case. Anyway, for a side $s$ of $\Delta$ denote by $\lambda_s\colon\Delta\rightarrow \RR_{\geq0}$ the distance function from the line prolonging $s.$ Note that this extends to a linear function with integral gradient on $\RR^2$ supporting $\Delta.$
Define $F_\Delta\colon\Delta\rightarrow\RR_{\geq 0}$ by \begin{equation}
F_\Delta(p)=\min_s\lambda_s(p)
\label{eq_fdelz} 
 \end{equation}
In this formula the minimum is taken over all sides $s$ of $\Delta,$ and we may think of $\lambda_s(x,y)=a+ix+jy$ as of monomial ``$ax^iy^j$'' and of $\min$ as of summation. Therefore, $F_\Delta$ is a formal analogue of a polynomial. This mode of thinking is casual for tropical geometry where such piece-wise linear polynomials are seen as the limits of the usual ones.

\begin{proposition}
$F_\Delta$ defines $C_\Delta.$
\label{prop_fdefc}
\end{proposition}

The verb ``defines'' means that $C_\Delta$ is the corner locus of $F_\Delta.$ Remark \ref{rem_initdist} justifies the proposition in the neighborhood of $\partial\Delta.$ To extend this, we note that $F_\Delta$ is a continuous concave piecewise-linear function. Consider the maximum ${\bold m}_\Delta$ of $F_\Delta$ and for $0<\e<{\bold m}_\Delta$ there exist a polygon $\Delta_\e$ obeying $\partial\Delta_\e=F_\Delta^{-1}(\e).$ The vertices of $\Delta_\e$ are the positions of particles at time $\e.$ 

\begin{thm}
The polygon $\Delta_\e$ is Delzant for $\e\in[0,{\bold m}_\Delta).$
\label{thm_deltedlz}
\end{thm}

 Note that $F_\Delta^{-1}({\bold m}_\Delta)$ has empty interior and cannot be Delzant.

\begin{proof}
Consider the smallest time $\e>0$ such that a pair of particles emitted from the ends of a side $s$ collide at point $p.$ If $s$ is removable then we can consider $\tilde\Delta$ the moment polygon of the blowdown (see Figure \ref{fig_blowdownc}). In this case either $\tilde\Delta_\e=\Delta_\e$ if $\e\geq e_0$ or $\Delta_\e$ is the blowup of $\tilde\Delta_\e$ otherwise. In both cases $\Delta_\e$ is Delzant whenever $\tilde\Delta_\e$ is Delzant, so we can use induction on the number of sides.

We are going to prove that if $s$ is not removable then $F_\Delta(p)=\e$ is equal to the maximum ${\bold m}_\Delta=\max_\Delta F_\Delta.$ In this case for $0\leq\delta<{\bold{m}_\Delta}$ the polygons $\Delta_\delta$ all have the same dual fan.

Consider the sides $s_1$ and $s_2$ of $\Delta$ adjacent to $s.$ There are three different situations (see Figure \ref{fig_collint}). In the first situation the sides $s_1$ and $s_2$ are parallel and we are done here by by the following Lemma.  

\begin{lemma}
Consider a point $p\in\Delta^\circ$ such that $F_\Delta(p)=\lambda_{s_1}(p)=\lambda_{s_2}(p)$ for a pair of parallel sides $s_1$ and $s_2.$ Then $F_\Delta(p)={\bold m}_\Delta.$
\label{lem_parlmax}
\end{lemma}

\begin{proof}
We note that $\Delta_\delta\subset\Delta$ is got squished (i.e. diameter is bounded and the area goes to zero) as $\delta\rightarrow F_\Delta(p)$ since the distance between a pair of parallel sides tends to zero. This is not possible unless $F_\Delta(p)$ is the maximal value.  
\end{proof}

\begin{figure}[h]

\begin{tikzpicture}

\draw(4,4)--++(0,-4);
\draw(4,4)--++(-2,0);
\draw(4,0)--++(-2,0);
\draw(4,4)--++(-1,-2)--++(1,-2);
\draw(4.2,2)node{$s$};
\draw(3,0.2)node{$s_2$};
\draw(2.7,2)node{$p$};
\draw[->](3.5,3)--++(-0.5,-1);
\draw[->](3.5,1)--++(-0.5,1);
\draw(3,4.2)node{$s_1$};

\begin{scope}[xshift=140]
\draw(4,4)--++(0,-4);
\draw(4,4)--++(-2,-1);
\draw(4,0)--++(-2,0);
\draw(4,4)--++(-1,-2)--++(1,-2);
\draw(4.2,2)node{$s$};
\draw(3,0.2)node{$s_2$};
\draw(2.7,2)node{$p$};
\draw[->](3.5,3)--++(-0.5,-1);
\draw[->](3.5,1)--++(-0.5,1);
\draw(2.8,3.7)node{$s_1$};
\end{scope}

\begin{scope}[xshift=280]
\draw(4,4)--++(0,-4);
\draw(4,4)--++(-2,1);
\draw(4,0)--++(-2,0);
\draw(4,4)--++(-1,-2)--++(1,-2);
\draw(4.2,2)node{$s$};
\draw(3,0.2)node{$s_2$};
\draw(2.7,2)node{$p$};
\draw[->](3.5,3)--++(-0.5,-1);
\draw[->](3.5,1)--++(-0.5,1);
\draw(2.8,4.2)node{$s_1$};
\end{scope}

\end{tikzpicture}

\caption{The possible configurations of sides with respect to a pair of collided particles. The collision happens at $p$ to the left from the vertical side $s.$ In the first case the adjacent sides $s_1$ and $s_2$ are parallel, in the second case their prolongations intersect to the left from $s$ or to the right in the third case.}
\label{fig_collint}
\end{figure}
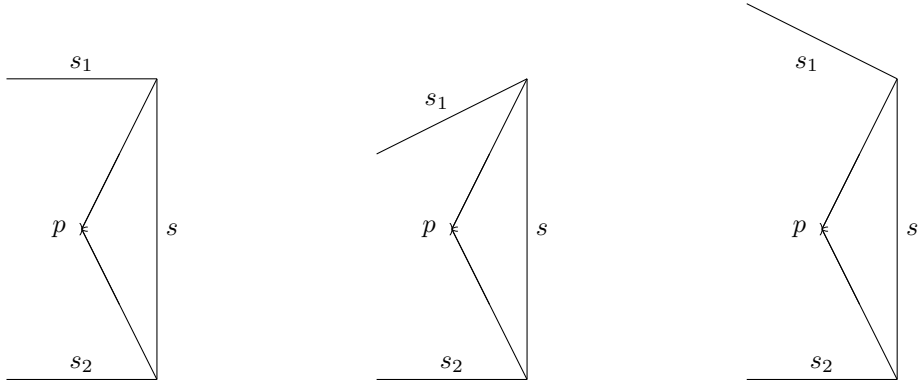

In the second case, the point of intersection of the prolongations of $s_1$ and $s_2$ is on the same side with respect to $s$ as the point of collision $p.$ If we think of $F_\Delta,$ its gradient $\nabla F_\Delta$ is well defined on the compliment to the tropical curve of $F_\Delta.$ Consider a side $s_3\neq s_1,s_2,s,$ the gradient of it support function $\lambda_{s_3}$ is the primitive vector orthogonal to a side $s_3$ and therefore is a positive combination of gradients for $\lambda_{s_1}$ and $\lambda_{s_1}.$ In particular, $p$ is a global maximum of $F_\Delta.$

In the third case, we see that $\nabla\lambda_s$ is a positive linear combination of $\nabla\lambda_{s_1}$ and $\nabla\lambda_{s_2}.$ Also, we know that $\nabla\lambda_{s_1},\nabla\lambda_{s}$ and $\nabla\lambda_{s_2},\nabla\lambda_{s}$ give a basis in $\ZZ^2.$ Now, note that the convex hull of the points $\nabla\lambda_{s_1},\nabla\lambda_{s_2},0,\nabla\lambda_{s}$ is a four-gone with area $1.$ It contains a triangle with vertices $\nabla\lambda_{s_1},\nabla\lambda_{s_2},0.$ (this is true since $s$ is getting shrunk, see Figure \ref{fig_srinks})
\begin{figure}

\begin{tikzpicture}
\begin{scope}[xshift=50]

\draw(4,4)--++(0,-4);
\draw(4,4)--++(-1.5,1.5);
\draw(4,0)--++(-4.5,0);
\draw(4,4)--++(-4,0)--++(4,-4);
\draw(4.2,2)node{$s$};
\draw(1,0.2)node{$s_2$};
\draw(0,4.2)node{$p$};
\draw(3,4.6)node{$s_1$};
\draw(4.3,0.2)node{$p_2$};
\draw(4.3,4)node{$p_1$};

\end{scope}

\begin{scope}[xshift=250,yshift=60,scale=1.5]

\draw[->](0,0)--++(1,0);
\draw(1.3,0.3)node{$-\nabla \lambda_s$};

\draw[dashed](0,-1)--++(1,2);

\draw[->](0,0)--++(1,1);
\draw(1,1.2)node{$-\nabla \lambda_{s_1}$};

\draw[->](0,0)--++(0,-1);
\draw(-0.4,-0.8)node{$-\nabla \lambda_{s_2}$};

\draw(-0.1,0.1)node{$0$};

\end{scope}

\end{tikzpicture}
\caption{The third type of collision and its local dual picture. Note that $\nabla \lambda_s$ doesn't belong to the triangle spaned by $\nabla \lambda_{s_1}$, $\nabla \lambda_{s_2}$ and $0$ because a segment $[p,p_i],$   $i=1,2,$ is orthogonal to the vector $\nabla\lambda_{s_i}-\nabla\lambda_{s_i}$ and $p\in\Delta.$}
\label{fig_srinks}
\end{figure}
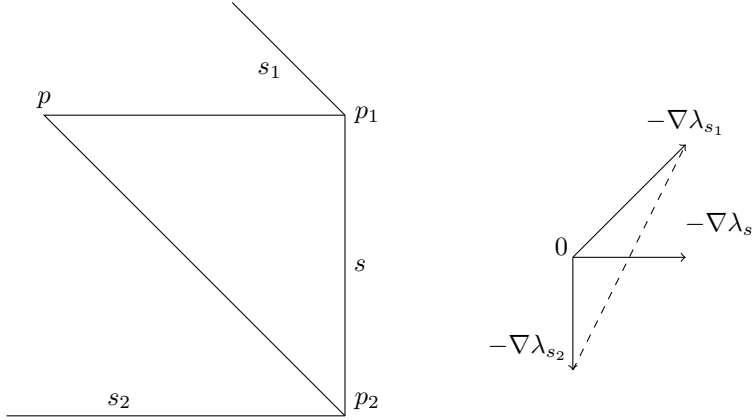

Therefore, this lattice triangle has the minimal possible area $1\over 2$ and $s$ must be removable since \begin{equation}
 \nabla\lambda_{s}=\nabla\lambda_{s_2}+\nabla\lambda_{s_2}.    
\label{eq_sumrem}
\end{equation}
 
\end{proof}

Summarizing the proof above, we saw that only the right configuration on Figure \ref{fig_collint} can give a collision at a non-terminal point. In fact, exactly two particles can meet at a non-maximal point. 

\begin{lemma}
If when passing from $\Delta$ to $\Delta_{\e}$ two sides collapse at the same vertex $p$ of $C_\Delta$ then $F_\Delta(p)={\bold m}_\Delta.$ 
\label{lem_moreonecol}
\end{lemma}

\begin{proof}
Supposing $p$ is not maximal, note that for both contracting sides we have a situation described on Figure \ref{fig_srinks}. Also, note that we can assume that these sides are adjacent. Now we apply $SL(2,\ZZ)$ transformation to make the contracting sides horizontal and vertical. In terms of Figure \ref{fig_srinks} we have fixed the slopes of $s$ and $s_2$ which are contracting now. There is only one position (on the dual picture) where we can add $-\nabla \lambda_{s_3}$ for $s_3$ adjacent to $s_2,$ i.e. (as in the end of the proof of Theorem \ref{thm_deltedlz}) $\nabla \lambda_{s_2}=\nabla \lambda_{s}+\nabla \lambda_{s_3}.$ Therefore, $s_1$ is parallel to $s_3$ and we arrive to a contradiction with Lemma \ref{lem_parlmax}.
\end{proof}

\begin{proof}[Proof of Proposition \ref{prop_fdefc}]
We use induction on time. Consider again the first time of a collision $\e.$ 
By Remark \ref{rem_initdist}, $C_\Delta$ coincides with the curve given by $F_\Delta$ in $\Delta\backslash\Delta_{\e}$. And this is enough for the proof if $\e={\bold m}_\Delta.$ Otherwise, by Theorem \ref{thm_deltedlz} we can run the particle process for $C_{\Delta_\e}.$ On the other hand, by Lemma \ref{lem_moreonecol} and equation \eqref{eq_sumrem} there is only one local model for a non-maximal collision shown on Figure \ref{fig_blowdownc} wich guaranties that the processes for $\Delta$ after the time $\e$ and the process for $\Delta_\e$ are the same. 
\end{proof}

The representation of $C_\Delta$ as a curve of $F_\Delta$ makes Proposition \ref{prop_lastedge} and Lemma \ref{lem_termination} evident. Indeed, the last edge $last_\Delta$ becomes just $F_\Delta^{-1}({\bold m}_\Delta)$ and for a point $p\in C_\Delta$ the value $F_\Delta(p)$ is the time at which $p$ is riched by a particle (or a group of particles if $p$ is a vertex). There are no multiple edges except for the maximal edge $last_\Delta$ since a non-maximal collision has a unique model shown on Figure \ref{fig_blowdownc}. Finally, note that the last edge has multiplicity two since it appears as a limit of a collapsing polygon $\Delta_\e$ with a pair of parallel sides and the gradients of the support functions of these sides are primitive and opposite.

\subsection{Tropical series of convex domain}
The formula \eqref{eq_fdelz} works in a greater generality.
For a compact convex domain $\Omega\subset\mathbb{R}^2$ define a function $F_\Omega\colon \Omega \rightarrow \mathbb{R}$  
\begin{equation} \label{eq_fomega}
 F_\Omega(z)=\inf_{v\in\mathbb{Z}^2} (\alpha_v+v\cdot z),
 \end{equation}
 where a number $\alpha_v$ for a vector $v\in\mathbb{Z}^2\backslash \{(0,0)\}$ is given by 
\begin{equation} \label{eq_alphav}
 \alpha_{v}=-\min_{z\in\Omega}z\cdot v.
 \end{equation}

\begin{remark}
In fact, $F(z)$ equals to the minimal normalized distance from $z$ to $\Omega$ in the following sense: for each primitive lattice vector $(p,q)$ we measure the distance between $z$  and the support line for $\Omega$, corresponding to $(p,q)$ and multiply this distance by $\sqrt{p^2+q^2}$. Then, $F(z)$ is equal to the minimal among all these numbers when $(p,q)$ runs over all primitive lattice vectors.
\end{remark}
 
 We call $F_\Omega$ the tropical series of $\Omega.$ The terminology comes from the fact that if in \eqref{eq_fomega} we replace $\mathrm{inf}$ with $\sum,$ summation with multiplication and the scalar product with an $\mathrm{exp}$ we would get something of the shape $\sum \alpha_v z^v$ which is a Laurent power series in two variables. Such a shift in arithmetics is custom for tropical geometry \cite{BIMS}. Following the analogy with analysis, we can say that $\Omega$ is the domain of convergence for the series representing the function $F_\Omega.$ 
 
\begin{remark}
If $\Omega$ is a polygon with rational slopes then only a finite number of monomials contribute to $F_\Omega,$ i.e. it can be represented by a tropical polynomial $\min_{v\in A}(a_v+z\cdot v)$ for a finite subset $A$ in $\ZZ^2\backslash\{(0,0)\}.$
\end{remark}

 To visualize $F_\Omega$ we look on its corner locus $C_\Omega\subset\Omega,$ i.e. the locus where $F_\Omega$ is not locally smooth. As a set, $C_\Omega$ is formed by a locally-finite union of segments with rational slopes, moreover, it has a natural structure of tropical analytic curve \cite{us_series,BIMS}. Note that $F_\Omega$ is linear on every face (i.e. on a connected component of $\Omega\backslash C_\Omega$),  and every edge of $C_\Omega$ has a prescribed multiplicity identified with the integral normalized length of the vector connecting $v_1$ and $v_2,$ where $v_j\cdot p+a_j$ represent the restrictions of $F_\Omega(p)$ to the faces adjacent to the edge. With this multiplicities we have the balancing condition on slopes satisfied at every vertex (see Figure \ref{fig_balancing}). To every vertex $p$ we associate a lattice polygon spanned by the monomials contributing at the vertex and $\mu(p)\in\mathbb{Z},$ {\it the multiplicity of the vertex} $p$  is defined to be twice the area of this dual polygon, see Figure \ref{fig_balancingdual} for a more detailed explanation.

 By definition, a vertex (or an edge) is called {\it smooth} if it has multiplicity one. 
 
 \begin{remark}
 Two vertices (or one vertex if $\lambda_\Omega=0$) where $F_\Omega$ attains its maximum are never smooth, since the origin, corresponding to the constant monomial contributing along the maximum, is never a vertex of the dual subdivision to $C_\Omega$. The conceptual reason is that $C_\Omega$ can be viewed as an elliptic curve and its genus is contracted to the maximal segment (or vertex). In particular, $C_\Omega\cap \Omega^\circ$ has no loops except for this hidden one, and thus is a tree. 
 \end{remark}

\begin{proposition}[\cite{us_series}] $F_\Omega$ is a non-negative continuous concave piecewise-linear function. The level set $F_\Omega^{-1}(\e)$ is $\partial\Omega$ for $\e=0,$ a segment with rational slope or a point for $\e={\bold m}_\Omega$, and a $\QQ$-polygon for intermediate values. 
 \label{prop_basicsF}
\end{proposition} 

Therefore, we can speak of a canonical approximation by polygons for an arbitrary convex domain. Moreover, we have the following.

\begin{thm}
If $\partial\Omega$ has no corners then the intermediate level sets of $F_\Omega$ are Delzant polygons.
\label{thm_omegldelz}
\end{thm}

\begin{proof}
For a general $\Omega$ the curve $C_{\Omega_\epsilon}$ is supported on a tree for $\bold{m}_\Omega>\epsilon>0,$ where $\bold{m}_\Omega$ denotes the maximal value of $F_\Omega$. Therefore, $\Omega_\epsilon$ is not Delzant if and only if  $C_{\Omega}$ has a non-maximal multiple edge. In particular, there exist at least three monomials of $F_\Omega$ contributing along such edge. Their corresponding supporting lines intersect along a corner of $\partial\Omega.$
\end{proof}

Note that, although it is less instructive, this short proof works for Theorem \ref{thm_deltedlz} as well. We prefer to keep the longer proof in section \ref{sec_curvepol} since its longer surrounding narrative serves as a friendly introduction to tropical curves.
 
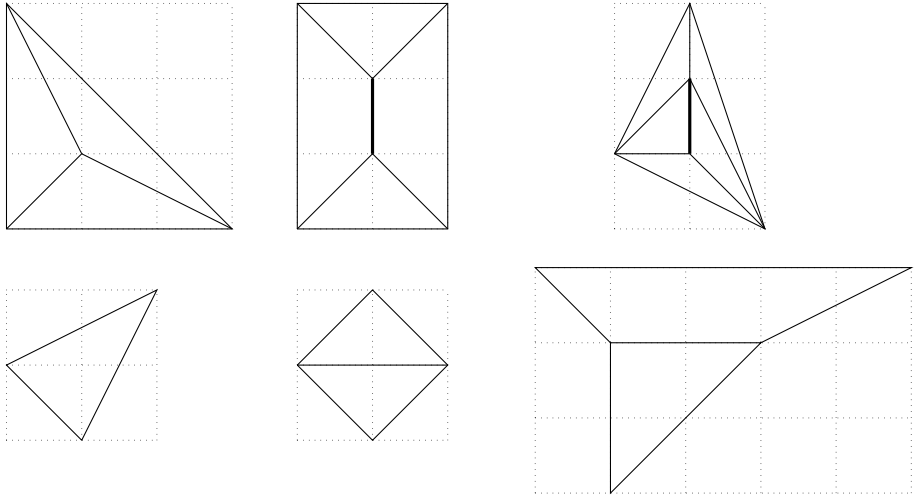
\begin{figure}
\begin{tikzpicture}
\draw[black,step = 1.0,very thin, dotted](0,0) grid (3,3);
\draw(0,0)--++(0,3)--++(3,-3)--++(-3,0);
\draw(1,1)--++(-1,-1);
\draw(1,1)--++(-1,2);
\draw(1,1)--++(2,-1);

\begin{scope}[yshift=-80]
\draw[black,step = 1.0,very thin, dotted](0,0) grid (2,2);
\draw(1,0)--++(1,2)--++(-2,-1)--++(1,-1);
\end{scope}

\begin{scope}[xshift=110]
\draw[black,step = 1.0,very thin, dotted](0,0) grid (2,3);
\draw(0,0)--++(0,3)--++(2,0)--++(0,-3)--++(-2,0);
\draw(0,0)--++(1,1)--++(1,-1);
\draw(0,3)--++(1,-1)--++(1,1);
\draw[very thick](1,1)--++(0,1);

\begin{scope}[yshift=-80]
\draw[black,step = 1.0,very thin, dotted](0,0) grid (2,2);
\draw(0,1)--++(1,1)--++(1,-1)--++(-1,-1)--++(-1,1);
\draw(0,1)--++(2,0);
\end{scope}

\end{scope}

\begin{scope}[xshift=230]
\draw[black,step = 1.0,very thin, dotted](0,0) grid (2,3);
\draw(0,1)--++(1,2)--++(1,-3)--++(-2,1);
\draw(0,1)--++(1,0);
\draw[very thick](1,1)--++(0,1);
\draw(1,2)--++(0,1);
\draw(1,1)--++(1,-1);
\draw(1,2)--++(1,-2);
\draw(0,1)--++(1,1);

\begin{scope}[yshift=-100,xshift=-30]
\draw[black,step = 1.0,very thin, dotted](0,0) grid (5,3);
\draw(1,2)--++(-1,1)--++(5,0)--++(-2,-1)--++(-2,-2)--++(0,2)--++(2,0);
\end{scope}

\end{scope}

\end{tikzpicture}

\caption{Examples of $\Omega$ with $C_\Omega$ inside. The corresponding seeds $\delta_\Omega$ are shown in the second row. Note that the right one is not convex, so the multiplicity two (thick) edge $F^{-1}_\Omega(\bold{m}_\Omega)$ has always non-zero length for domains with such a seed.}

\label{fig_seeds}
\end{figure}

We would like to capture {\it a type of singularity} along the segment $F_\Omega^{-1} ({\bold m}_\Omega)$.   Denote by $\delta_\Omega$ the {\it seed} of the tree $C_\Omega,$ defined as the union of dual polygons to the ends of the maximal segment, see Figure \ref{fig_seeds}. Note that the segment can be degenerate, in this case $\delta_\Omega$ is just the  (convex) lattice polygon dual to the vertex $p$ where the maximum is attained. The following proposition is straightforward.
 
 \begin{proposition}
 The origin is the only lattice point in the interior of $\delta_\Omega.$
 \end{proposition}

The space of all convex domains is therefore stratified with respect to the type of singularity of $C_\Omega$ along the maximum of $F_\Omega.$ The strata are enumerated by different seeds $\delta_\Omega$. Each stratum is parametrized by the lattice lengths of edges of $C_\Omega.$ 

\begin{remark}
The seed $\delta_\Omega$ and the moduli of $C_\Omega$ determine $\Omega$ up to parallel translations.
\label{rem_convexcoordinates}
\end{remark}

We can interpret the statement of Remark \ref{rem_convexcoordinates} as an identification of the space of convex domains with a certain moduli space of infinetly punctured tropical curves. Practically, this means that the lengths of the edges in $C_\Omega$ are the complete coordinates for $\Omega,$ and therefore its invariants, such as lattice perimeter or area, have to be expressible in terms of these coordinates. The formula \eqref{eq_main2} is a manifestation of the principle in the case of $\Omega$ equal to a disk. 

\begin{remark}
If $p_1$ and $p_2$ are adjacent vertices of $C_\Omega$ such that $F_\Omega(p_1)<\bold{m}_\Omega$ then $$\operatorname{Length}_\ZZ [p_1,p_2]=|F_\Omega(p_1)-F_\Omega(p_2)|,$$ where $[p_1,p_2]$ denotes the segment of $C_\Omega$ connecting the vertices. \end{remark}

\subsection{Tropical series of the unit disk}
This section gives one more explanation of the Example \ref{ex_pi}.
 \begin{figure}[h]
\begin{minipage}{0.45\textwidth}
\includegraphics[width=\linewidth]{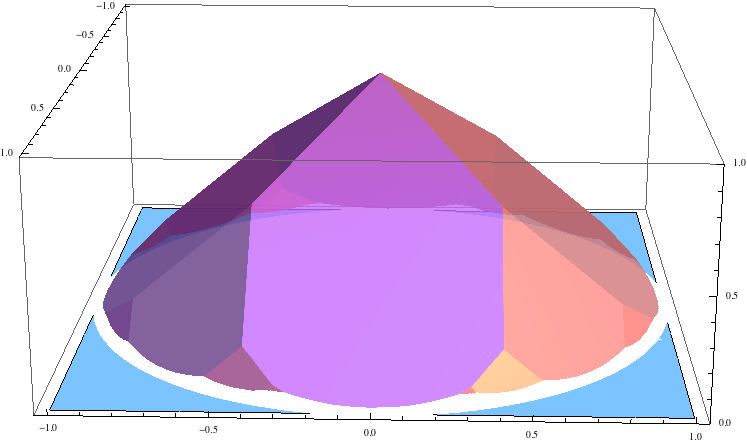}
\end{minipage}
\hfill%
\begin{minipage}{0.45\textwidth}\raggedleft
\includegraphics[width=\linewidth]{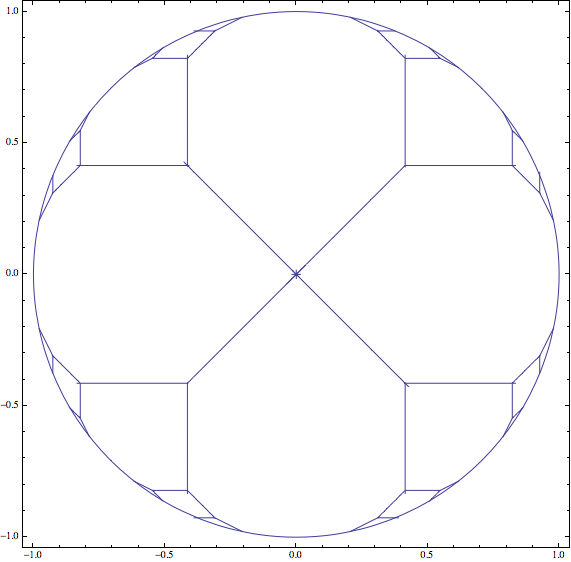}
\end{minipage}
\caption{A plot of the series $F_\Omega$ and its tropical analytic curve $C_\Omega$ for $\Omega$ a disk.}
\label{fig_circlecurve}
\end{figure}
Let $\Omega$ be the unit disk $\{x^2+y^2\leq 1\}$.  Applying equations \eqref{eq_alphav} and \eqref{eq_fomega}, $F$ is given by

$$ F(z)=\min_{v\in \mathbb{Z}^2} (v\cdot z+|v|).$$

 The graph of $F$ and its corner locus $C$ are shown on the Figure \ref{fig_circlecurve}. Notice that only four {\it tropical monomials} contribute to $F$ at the origin. These are $1+x,1-x,1+y,1-y$ and the maximum of $F$ is $1.$ The monomials give an X part of $C$ near the origin where the maximum is taken.

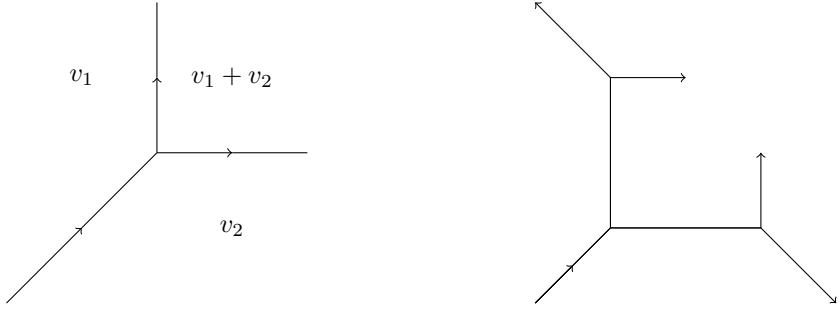
\begin{figure}[h]
\begin{tikzpicture}

\draw(0,0)[->]--++(1,1);
\draw(1,1)--++(1,1);
\draw(2,2)[->]--++(1,0);
\draw(2,2)[->]--++(0,1);
\draw(3,2)--++(1,0);
\draw(2,3)--++(0,1);
\draw(3,3)node{$v_1+v_2$};
\draw(3,1)node{$v_2$};
\draw(1,3)node{$v_1$};

\begin{scope}[xshift=200]
\draw(0,0)[->]--++(0.5,0.5);
\draw(0,0)[->]--++(1,1)--++(2,0)--++(1,-1);
\draw(1,1)[->]--++(2,0)--++(0,1);
\draw(0,0)[->]--++(1,1)--++(0,2)--++(-1,1);
\draw(1,1)[->]--++(0,2)--++(1,0);
\end{scope}

\end{tikzpicture}
\caption{On the left: branching pattern for the tropical curve at the vertex $z(v_1,v_2)$, the arrow denotes the direction towards the boundary. The vectors $v_1,v_2,v_1+v_2$ are the values of the gradient of $F$ in the complement to $C,$ its edges are orthogonal to $v_1-v_2,v_1$ and $v_2.$   On the right: a schematic picture for two more branchings of $C,$ the vertices are $z(v_1,v_2),$ $z(v_1,v_1+v_2)$ and $z(v_1+v_2,v_2).$ The parallel segments of the drawing represent parallel edges of $C.$ Compare this with Figures \ref{fig_blowdownc} and \ref{fig_circlecurve}.}
\label{fig_newvertex}
\end{figure}

Clearly, $C$ is a tree and, apart from the origin, all the vertices of $C$ are trivalent. Starting from the origin, the tree goes in four directions and stratas branching infinitely many times. A branching gives a vertex and the third monomial adjacent to the new vertex is chosen in a systematic manner shown in the Figure \ref{fig_newvertex}.  Each such vertex $z(v_1,v_2)$ is adjacent to the three regions in the complement to $C$ and on these regions $F$ is represented by the monomials $|v_1|+z\cdot v_1,$ $|v_2|+z\cdot v_2$ and $|v_1+v_2|+z\cdot (v_1+v_2),$ where $v_1$ and $v_2$ are vectors from the same quadrant in $\mathbb{R}^2$ forming the basis of $\mathbb{Z}^2.$ The value of $F$ at $z(v_1,v_2)$ is given by \begin{equation} F(z(v_1,v_2))=|v_1|+|v_2|-|v_1+v_2|\label{eq_seriesatz}
\end{equation} which is exactly the error to the triangular inequality for the primitive triangle formed by vectors $v_1,v_2,v_1+v_2.$ The only multiple vertex of $C_\Omega$ is the origin, where the maximum $\bold{m}_\Omega=1$ is attained and the multiplicity is equal to $4$ since $\delta_\Omega$ is a union of $4$ triangles with area $1/2.$ 

%\begin{figure}
%\begin{tikzpicture}[scale=0.3]
%
%\draw[black,step = 2.0, very thin, dotted](-3,-9) grid (9,0);
%\draw[dashed](2,0)--++(-2,-4);
%\draw(0,-4)--++(-2,-4);
%\draw[dashed](2,0)--++(2,0);
%\draw(4,0)--++(4,0);
%\draw(4,0)--++(-4,-4);
%
%
%\begin{scope}[yshift=-300]
%
%\draw(0,-3)--++(0,-3);
%\draw(0,-3)--++(6,-3)--++(-6,0);
%\draw[black,step = 3.0,very thin, dotted](-3,-9) grid (9,0);
%\draw[dashed](0,-3)--++(3,-3);
%\end{scope}
%
%\begin{scope}[xshift=400]
%
%\draw[black,step = 2.0,very thin, dotted](-4,-10) grid (8,-1);
%\draw(-4,-8)--++(2,2);
%\draw[dashed](-2,-6)--++(2,2);
%\draw[dashed](0,-4)--++(4,-2);
%\draw(4,-6)--++(4,-2);
%\draw(4,-6)--++(-6,0);
%
%\begin{scope}[yshift=-300]
%\draw(3,0)--++(-3,-6);
%\draw(3,0)--++(3,-3)--++(-3,0)--++(-3,-3);
%\draw[black,step = 3.0,very thin, dotted](-3,-9) grid (9,0);
%\draw[dashed](3,0)--++(0,-3);
%\end{scope}
%
%\end{scope}
%
%
%\begin{scope}[xshift=800]
%
%\draw[black,step = 2.0,very thin, dotted](-3,-9) grid (9,0);
%\draw[dashed](0,0)--++(-2,-6);
%\draw[dashed](0,0)--++(6,-2);
%\draw(-2,-6)--++(-1,-3);
%\draw(6,-2)--++(3,-1);
%\draw(-2,-6)--++(1,2)--++(2,2)--++(5,0);
%
%\begin{scope}[yshift=-300]
%\draw(0,0)--++(-3,-9);
%\draw(0,0)--++(9,-3)--++(-9,0)--++(-3,-6);
%\draw[black,step = 3.0,very thin, dotted](-3,-9) grid (9,0);
%\draw[dashed](0,0)--++(0,-3);
%\draw[dashed](0,0)--++(3,-3);
%\draw[dashed](0,0)--++(6,-3);
%\end{scope}
%
%\end{scope}
%
%\end{tikzpicture}
%
%\end{figure}

%%it is better to have more space before the formulae
\medskip
%%

%\subsection{Symplectic length and complex curves}
%Tropical geometry is related to complex geometry in a similar way as particles are related to strings. Oversimplifying, one can replace a particle with a loop. The world sheet of a string then is a holomorphic curve which retracts to a tropical curve. We can measure the size of a string as a symplectic area of the underlying topological surface. The principle features of the symplectic area are that it is invariant under deformations of the string and survives under tropicalization giving the {\it symplectic length} of a tropical curve. 
%
%\subsection{Sandpiles}
% 
%\section{Tropical curve of convex domain}
%\label{sec_curvedom}
%Approximations by polygons from inside, moduli of tropical curve and coordinates on domains, topology. 

\subsection{Lattice, Euclidean and symplectic perimeter: conservation laws}
\label{sec_lateuclsympl}
While cutting the corners of a polygon (going from right to left on Figure \ref{fig_blowdownc}) one can observe that certain quantities are preserved. Most significantly we have
$$\Length_\ZZ \partial\Omega+\Length_\ZZ C_\Omega=12 \bold{m}_\Omega+4 \Length_\ZZ F_\Omega^{-1}(\bold{m}_\Omega).$$

There is a similar and simpler identity for the symplectic length (\cite{yuArea,us_series}) $$\Length_s \partial\Omega-\Length_s C_\Omega=0.$$ It follows from the deformation invariance of the symplectic length and from the fact that $C_\Omega$ can be deformed on $\partial\Omega.$ In the case when $\Omega$ is a $\mathbb{Q}$-polygon it is finite and furthermore the numbers of vertices of $C_\Omega$ and of $\partial\Omega$ coincide if we count them with multiplicities. 
 
\subsection{Universal formulas} 
\label{sec_uniformls}
For $(v_1 v_2)\in SL(2,\ZZ)$ define $f(v_1v_2)$ to be equal to $F_\Omega(z(v_1,v_2))$ if there exists  a vertex $z(v_1,v_2)$ of  $C_\Omega$ where monomials $v_1,v_2$ and $v_1+v_2$ contribute and $f(v_1v_2)=0$ if such a vertex
 doesn't exist. Denote by $\hat\Omega$ the minimal model of $\Omega,$ i.e. the polygon in which $\Omega$ is inscribed with an unbranched tropical curve $C_{\hat\Omega}$ coinciding with $C_\Omega$ near their maximal edges. For example, the minimal model of a disk is a square.  Then $\Omega$ is the result of corner cuts (exhaustion), each cut has a shape of Delzant triangle of the size $f(v_1v_2).$ We have the following
  
$$\Length_\ZZ \partial \Omega=\Length_\ZZ \partial\hat \Omega-\sum_{(v_1v_2)\in SL_2\ZZ} f(v_1v_2)$$

$$\Length \partial \Omega=\Length \partial \hat \Omega-\sum_{(v_1v_2)\in SL_2\ZZ} f(v_1v_2) (|v_1|+|v_2|-|v_1+v_2|)$$

$$\Length_s \partial \Omega=\Length_s \partial \hat \Omega+2\sum_{(v_1v_2)\in SL_2\ZZ} f(v_1v_2)(v_1\cdot v_2)$$

$$\operatorname{Area}  \Omega=\operatorname{Area} \hat \Omega-{1\over 2}\sum_{(v_1v_2)\in SL_2\ZZ} f^2(v_1v_2)$$

$$\int_\Omega F_\Omega=\int_{\hat\Omega}F_{\hat \Omega}-{1\over 6}\sum_{(v_1v_2)\in SL_2\ZZ} f^3(v_1v_2).$$

Note that in the case of the unit disk \eqref{eq_seriesatz} the formulas for the length of $\partial\Omega$ and for the area of $\Omega$ become identical.

\subsection{Example: disk in $L^\mu$-norm}

Consider the domain $\Omega_\mu$ be given by $|x|^\mu+|y|^\mu\leq 1$, for $\mu>1.$ Note that due to the symmetry under reflections the germ of $C_{\Omega_\mu}$ near the origin doesn't depend on $\mu.$ Therefore, all $\Omega_\mu$ have the same minimal model $\hat\Omega_\mu=[-1,1]^2.$ 

% Let us find a point where the tangent line to its boundary is $px+qy+c_{pq}=0$.

%$x^{\mu-1} dx+y^{\mu -1} dy=0$. $pdx+qdy=0$. $\frac{p}{q}=(\frac{x}{y})^{\mu-1}$. $y=(\frac{q}{p})^{\frac{1}{\mu-1}}x$. 

%$x^\mu+((\frac{q}{p})^{\frac{1}{\mu-1}}x)^{\mu}=1$. then

%$x^{\mu} (1+(\frac{q}{p})^{\frac{1}{\mu-1}})^\mu) =1 $

%$x=\frac{q^{\frac{\mu}{\mu-1}}}{(p^{\frac{1}{\mu-1}}+q^{\frac{1}{\mu-1}})^{\mu}}$

An easy computation shows that a quarter of the area of $\Omega_\mu$ is equal to $${{\Gamma^2(2-\nu^{-1})}\over{\Gamma(3-2\nu^{-1})}}=1-{1\over 2}\sum(|v_1|_\nu+|v_2|_\nu-|v_1+v_2|_\nu)^2,$$ where $|\cdot|_\nu$ denotes the dual norm, i.e. $\mu^{-1}+\nu^{-1}=1,$ and the sum runs over ${(v_1v_2)\in SL^+(2,\ZZ)}.$

\bibliography{../../bibliography}
\bibliographystyle{abbrv}

\end{document}